\documentclass[12pt,reqno]{amsart}
\usepackage[top=1in, bottom=1in, left=1in, right=1in]{geometry}

\usepackage{times, amsthm, amssymb, amsmath, amsfonts, graphicx, mathrsfs}
\usepackage{mathtools}
\usepackage{tikz} 
\usetikzlibrary{arrows, cd, matrix, positioning, decorations.markings}
\usepackage{xcolor, soul} 
\usepackage{bbm}
\usepackage{xypic}
\usepackage{caption}
\usepackage{subcaption}
\usepackage{mathscinet}
\usepackage{enumitem} 
\usepackage{blkarray, bigstrut}
\usepackage{arydshln}
\usepackage{multirow, multicol}

\newtheorem{theorem}{Theorem}[section]

\newtheorem{corollary}[theorem]{Corollary}

\theoremstyle{definition}
\newtheorem{definition}[theorem]{Definition}
\newtheorem{example}[theorem]{Example}
\newtheorem{construction}[theorem]{Construction}
 
\theoremstyle{remark}
\newtheorem{remark}[theorem]{Remark}

\newtheorem*{notation}{Notation}

\newcommand{\rank}{\operatorname{rank}}

\newcommand{\lcm}{\operatorname{lcm}}
\newcommand{\degree}{\operatorname{degree}}

\newcommand{\bbN}{\mathbbm{N}}

\newcommand{\bbk}{\mathbbm{k}}

\newcommand{\fraka}{{\mathfrak{a}}}
\newcommand{\frakm}{{\mathfrak{m}}}

\newcommand{\bfF}{{\mathbf{F}}}
\newcommand{\bfG}{{\mathbf{G}}}

\newcommand{\bfT}{{\mathbf{T}}}

\newcommand{\bfb}{{\mathbf{b}}}

\newcommand{\bfu}{{\mathbf{u}}}

\newcommand{\bfei}{{\mathbf{e}_i}}
\newcommand{\bfej}{{\mathbf{e}_j}}

\newcommand{\bfet}{{\mathbf{e}_t}}

\newcommand{\bfzero}{{\mathbf{0}}}

\newcommand{\barS}{{\overline{S}}}
\newcommand{\barSt}{{\overline{S \cup t}}}

\newcommand{\tildesigma}{{\widetilde{\sigma}}}

\newcommand{\eps}{{\varepsilon}}
\newcommand{\eS}{{\varepsilon_S}}
\newcommand{\eSt}{{\varepsilon_{S \cup t}}}
\newcommand{\eStq}{{\varepsilon_{S \cup t \cup q}}}

\newcommand{\Sq}{{S \cup q}}
\newcommand{\St}{{S \cup t}}

\newcommand{\Stq}{{S \cup t \cup q}}

\newcommand{\fij}{{f_{i,j}}}
\newcommand{\fit}{{f_{i,t}}}
\newcommand{\fiq}{{f_{i,q}}}
\newcommand{\fjt}{{f_{j,t}}}
\newcommand{\fjq}{{f_{j,q}}}

\newcommand{\ptS}{{p_{t,S}}}
\newcommand{\ptSq}{{p_{t,S \cup q}}}
\newcommand{\pqS}{{p_{q,S}}}
\newcommand{\pqSt}{{p_{q,S \cup t}}}

\newcommand{\oz}{{10}}
\newcommand{\zo}{{01}}

\newcommand{\barray}[2]{{\left[ \begin{array}{#1} #2 \end{array} \right]}}
\newcommand{\ideal}[1]{{\langle #1 \rangle}}

\tikzset{->-/.style={decoration={
  markings,
  mark=at position #1 with {\arrow{>}}},postaction={decorate}}}

\title{A Taylor Resolution Over Complete Intersections}
\author{Aleksandra Sobieska}
\date{\today}

\allowdisplaybreaks
\begin{document}

\tikzset{->-/.style={decoration={
  markings,
  mark=at position #1 with {\arrow{latex}}},postaction={decorate}}}

\maketitle

\begin{abstract}

The Taylor resolution is a fundamental object in the study of free resolutions over the polynomial ring, due to its explicit formula, cellular/combinatorial structure, and applicability to any and all monomial ideals. This paper generalizes the Taylor resolution to complete intersection rings via the Eisenbud--Shamash construction.

\end{abstract}

\section{Introduction}

The Taylor resolution \cite{tIdealsGenByMonomsIARSeq} universally resolves any monomial ideal over the polynomial ring; it generalizes the Koszul complex, has straightforward cellular support, and is eminently computable. Indeed, it is often among the first constructions one learns to build a free resolution. Its combinatorial and explicit nature have 
influenced many advances in commutative algebra, some of which can be found in  \cite{hhTightClosureInvariantThATBrianconSkodaTheorem, lANewExplFinFrResOIdGenByMonoms, ekMinResOSomeMonId, eCommAlgWAView, bsCellResOMonMod, 
bwDiscreteMorseThForCellRes, msCombCommAlg, pGradedSyzygies} and the references therein. 

Resolutions over singular rings are typically infinite and more complicated. Results for special types of rings -- complete intersections \cite{gAChangeORingThWAppToPoincareSAIntMult, gOTDeviatnsOALocalRing, eHomAlgOnACompleteInt, aModOFinVirtualProjDim, agpCompleteIntersectionDimension, abHomAlgModARegSeqWSpecialAttToCodimTwo, epMinFResOCompleteInt}, Golod rings \cite{gHomOSomeLocalRings, hrwCompLinIdealsAGolodRings, gpwRationalityFGenericToricR, bjOTGolodPOStanleyReisnerR}, Koszul algebras \cite{fKoszulAlgs, aeRegularityOModulesOAKoszulAlg, rOTCharOKoszulAlgFourCounterex} -- and/or special types of modules -- the ground field \cite{tHomONoethRingsALocalRings, pKoszulRes, mHomology, fDetOACOPoincareSeries, aACounterexampleTAConjOSerre, rsAToricRingWIrrationalPoincareBettiSeries} -- have been developed, but still much remains to be understood. 
Some surveys on infinite free resolutions can be found in \cite{aInfFRes, mpInfGrFRes}. 

In this article, we generalize the Taylor resolution to complete intersection rings. 
The fundamental idea behind this construction is to produce the homotopies from the Eisenbud--Shamash construction entirely in terms of combinatorial data (essentially ratios of lcms).
In providing an explicit formula for the system of homotopies, we show that no higher homotopies are necessary to write the complete resolution. 
Because the entries are easily expressed as ratios of $\lcm$s, this method generates a body of examples of infinite free resolutions, and provides combinatorial recipes for matrix factorizations. 
Finally, it is worth noting that, in addition to generalizing the Taylor resolution of monomial ideals to complete intersections, this construction recovers the well-known Tate resolution \cite{tHomONoethRingsALocalRings} as well as the resolution over monomial complete intersections in \cite{iFreeResAChangeORings}.

The main result can be found in Definition~\ref{def:taylorhomotopies} and Theorem~\ref{thm:systemofhomotopiesforT}. We provide an example of how this construction provides a direct way to resolve a monomial ideal over a complete intersection. 

\begin{example}
Let $Q = \bbk[x,y,z]$ and $I = \ideal{x^2, y^2, z^2}$, and $\fraka = \ideal{x^2z + xy^2}$ and $R=Q/\fraka$. The Taylor resolution of $I$ over $R$ given by Theorem~\ref{thm:systemofhomotopiesforT} is 
\[
\bfF: \cdots 
\xrightarrow{\varphi_5} 
\begin{array}{c} R(-6) \\ \oplus \\ R(-7)^3 \end{array}
\xrightarrow{\varphi_4} 
\begin{array}{c} R(-5)^3 \\ \oplus \\ R(-6) \end{array} 
\xrightarrow{\varphi_3} 
\begin{array}{c} R(-3) \\ \oplus \\ R(-4)^3 \end{array} 
\xrightarrow{\varphi_2} 
R(-2)^3 
\xrightarrow{\varphi_1} 
R 
\rightarrow 0
\]
where 
$F_{2i-1} = R(-3i+3) \oplus R(-3i+2)^3$ for $i \geq 2$ 
and 
$F_{2i} = R(-3i+1)^3 \oplus R(-3i)$, 
the initial maps are 
\[
\varphi_1 = 
    \begin{blockarray}{*{4}{>{\scriptstyle}c}<{}}
    1 & 2 & 3 \\ 
    \begin{block}{[*{3}r]>{\scriptstyle}c<{}}
    x^2 & y^2 & z^2 & \emptyset \\
    \end{block}
    \end{blockarray}
\text{ and }
\varphi_2 = 
    \begin{blockarray}{*{5}{>{\scriptstyle}c}<{}}
    \emptyset & 12 & 13 & 23 \\ 
    \begin{block}{[r|*{3}r]>{\scriptstyle}c<{}}
    z & y^2 & z^2 & 0 & 1 \\ 
    x & -x^2 & 0 & z^2 & 2 \\ 
    0 & 0 & -x^2 & -y^2 & 3 \\
    \end{block}
    \end{blockarray},
\] 
and the tail of the resolution is periodic with maps 
\[ 
\varphi_{2i-1} = 
    \begin{blockarray}{*{5}{>{\scriptstyle}c}<{}}
    1 & 2 & 3 & 123 \\ 
    \begin{block}{[*{3}r|r]>{\scriptstyle}c<{}}
    \color{blue}{x^2} & \color{blue}{y^2} & \color{blue}{z^2} & 0 & \emptyset \\
    \BAhhline{----}
    \color{red}{x} & \color{red}{-z} & 0 & \color{blue}{z^2} & 12 \\
    0 & 0 & \color{red}{-z} & \color{blue}{-y^2} & 13 \\ 
    0 & 0 & \color{red}{-x} & \color{blue}{x^2} & 23 \\
    \end{block}
    \end{blockarray}
\text{ and }
\varphi_{2i} = 
    \begin{blockarray}{*{5}{>{\scriptstyle}c}<{}}
    \emptyset & 12 & 13 & 23 \\ 
    \begin{block}{[r|*{3}r]>{\scriptstyle}c<{}}
    \color{red}{z} & \color{blue}{y^2} & \color{blue}{z^2} & 0 & 1 \\ 
    \color{red}{x} & \color{blue}{-x^2} & 0 & \color{blue}{z^2} & 2 \\ 
    0 & 0 & \color{blue}{-x^2} & \color{blue}{-y^2} & 3 \\
    \BAhhline{----}
    0 & 0 & \color{red}{-x} & \color{red}{z} & 123 \\
    \end{block}
    \end{blockarray}
\]
for $i \geq 2$.

We can see all the maps from the classic Taylor resolution inside the differentials of $\bfF$ (written in blue). As with the Taylor resolution, rows and columns indexed with subsets of $\{1,2,3\}$ corresponding to subsets of the generators of $I$, and there is a non-zero entry in the matrix factorization when the row label is a subset of the column label with just one element missing. 

The entries accounting for the singularity of $R$ are placed in prescribed spots (written in red) via the expression $x^2z + y^2x = z \cdot x^2 + x \cdot y^2$, which gives the generator of $\fraka$ as a combination of the first and second generators of $I$. Therefore, there is a non-zero entry in the matrix factorization when the row label is a union of the column label and either $\{1\}$ or $\{2\}$. The formula for the entry is a ratio of lcm's of subsets of generators of $I$, similar to the Taylor resolution. 

\begin{figure}
\centering 
  \begin{subfigure}[b]{0.47\linewidth}
  \centering
\begin{tikzpicture}
[scale=.65, vertices/.style={draw, fill=black, circle, inner sep=2pt}]

    \node (123) at (0,0) [vertices, label=left:$123$] {};
    \node (12) at (3,4) [vertices, label=above:$12$] {};
    \node (13) at (3,2) [vertices, label=below:$13$] {};
    \node (23) at (3,0) [vertices, label=below:$23$] {};
    \node (1) at (6,6) [vertices, label=above:$1$] {};
    \node (2) at (6,4) [vertices, label=above:$2$] {};
    \node (3) at (6,2) [vertices, label=below:$3$] {};
    \node (e) at (9,6) [vertices,label=$\emptyset$] {};
    \foreach \from/\to in {123/12,123/13,123/23, 1/e,2/e,3/e} 
        \draw[blue, ultra thick, ->-=.6] (\from) -- (\to);
    
    \foreach \from/\to in {12/1,12/2, 13/1, 13/3, 23/2,23/3}
        \draw[blue, ultra thick, ->-=.75] (\from) -- (\to);
\end{tikzpicture}
\caption{The Taylor resolution as a directed graph}\label{fig:taylorres}
\end{subfigure}
\begin{subfigure}[b]{0.47\linewidth}
\centering 
\begin{tikzpicture}
[scale=.75, vertices/.style={draw, fill=black, circle, inner sep=2pt}]

    \node (123) at (6,0) [vertices, label=below:$123$] {};
    \node (12) at (2,4) [vertices, label=above:$12$] {};
    \node (13) at (2,2) [vertices, label=below:$13$] {};
    \node (23) at (2,0) [vertices, label=below:$23$] {};
    \node (1) at (6,6) [vertices, label=above:$1$] {};
    \node (2) at (6,4) [vertices, label=above:$2$] {};
    \node (3) at (6,2) [vertices, label=below:$3$] {};
    \node (e) at (2,6) [vertices,label=$\emptyset$] {};
    \foreach \from/\to in {123/12,123/13,123/23, 1/e,2/e,3/e} 
        \draw[blue, ultra thick, ->-=.85] (\from) -- (\to);
    
    \foreach \from/\to in {12/1,12/2, 13/1, 13/3, 23/2,23/3}
        \draw[blue, ultra thick, ->-=.85] (\from) -- (\to);
    
    \draw[red, ultra thick, ->-=.6] (e) to [bend left=20] (1);
    \draw[red, ultra thick, ->-=.6] (2) to [bend right=20] (12);
    \draw[red, ultra thick, ->-=.6] (3) to [bend left=20] (13);
    \draw[red, ultra thick, ->-=.6] (23) to [bend right=20] (123);
    
    \draw[red, ultra thick, ->-=.6] (e) to [bend left=20] (2);
    \draw[red, ultra thick, ->-=.6] (1) to [bend right=20] (12);
    \draw[red, ultra thick, ->-=.7] (3) to [bend left=20] (23); 
    \draw[red, ultra thick, ->-=.7] (13) to [bend right=20] (123);
    
\end{tikzpicture}
\caption{The matrix factorization as a directed graph}\label{fig:matfactorizn}
\end{subfigure}
\caption{A graphical description of the matrix factorization}
\end{figure}
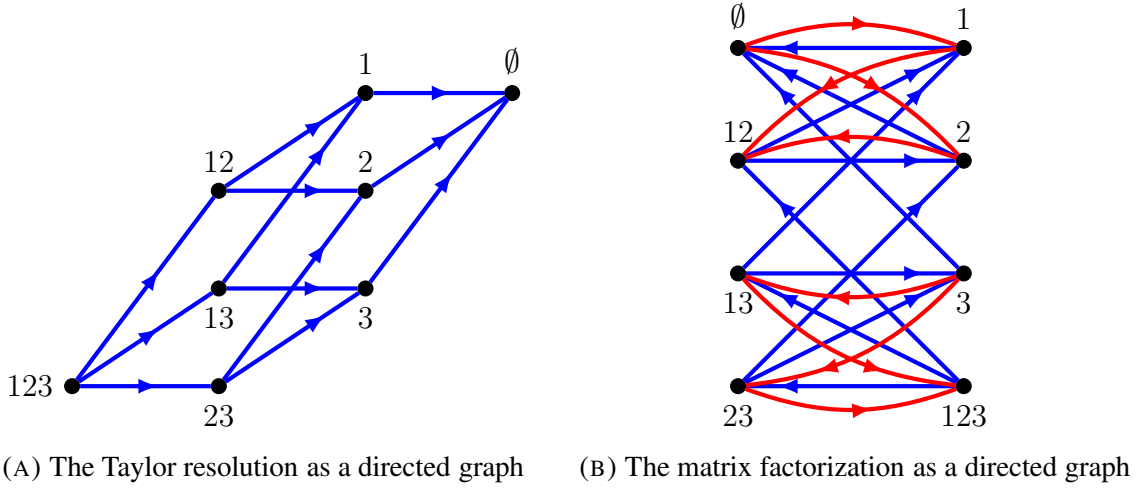

Figure~\ref{fig:taylorres} shows the summands (represented by vertices) and non-zero differential arrows (represented by blue arrows) from the Taylor resolution. In Figure~\ref{fig:matfactorizn}, the original directed graph is rearranged into a bipartite graph with even-cardinality subsets on one side and odd-cardinality subsets on the other. With the addition of the eight red arrows either from a vertex $S$ to $\{1\} \cup S$ or from $S$ to $\{2\} \cup S$, we can encapsulate the induced matrix factorization with a edge-colored, directed bipartite graph, where each edge represents a nonzero entry -- the blue edges give a ``Taylor'' entry and the red edges give a ``lift'' entry. 

\end{example}

\textit{Acknowledgements:} The author would like to thank Daniel Erman for many illuminating conversations during all stages of this article's development. Most, if not all, computations were aided by \texttt{Macaulay2} \cite{gsM2}.

\section{Background}

\begin{notation}
For a positive integer $r$, use $[r]$ to denote the set $\{1,2, \ldots, r\}$. We will occasionally identify the singleton set $\{s\}$ and $s$ when $s \in [r]$, especially when appearing in long and teeny subscripts.
\end{notation}

\subsection{The Taylor Resolution}

Let $Q = \bbk[x_1, \ldots, x_n]$ be a polynomial ring over a field and $I \subseteq Q$ be a monomial ideal $I = \langle m_1, \ldots, m_r \rangle$. For a set $S = \{j_1, \ldots, j_k\} \subseteq [r]$ with $j_1 < \ldots < j_k$, let $m_S = \lcm(m_{j_1}, \ldots, m_{j_k})$, and use $v_S$ to denote $\degree m_S$. 

\begin{definition}[\cite{tIdealsGenByMonomsIARSeq}]\label{def:taylorres} 
The Taylor resolution $Q/I$ for a monomial ideal $I = \
\langle m_1, \ldots, m_r \rangle$ is given by $(T_k, \tau_k)$, where 
\[T_k = \bigoplus\limits_{S \subseteq [r], \ |S| = k} Q(-v_S)\]
and 
\[\tau_k(\varepsilon_S) = \sum\limits_{s_i \in S} (-1)^{k-i} \frac{m_S}{m_{S-s_i}} \varepsilon_{S - s_i} \text{ for } S = \{s_1, \ldots, s_k\}.\]
\end{definition}

The resolution $\bfT$ is highly non-minimal, but is always a resolution for any monomial ideal $I$.

\example \label{ex:taylor} Let $Q = \bbk[x,y,z]$ and $I = \langle xy, xz, yz \rangle$. The Taylor resolution of $Q/I$ over $Q$ is given by 
\[
0 \rightarrow Q(-3) \xrightarrow{\tau_3} Q(-3)^3 \xrightarrow{\tau_2} Q(-2)^3 \xrightarrow{\tau_1} Q \rightarrow 0
\]
where we label and order the basis elements as follows: 
\begin{enumerate}
    \item basis for $T_0 = Q$: $\{ \eps_\emptyset \}$ 
    \item basis for $T_1 = Q(-2)^3$: $\{ \eps_1, \eps_2, \eps_3 \}$
    \item basis for $T_2 = Q(-3)^3$: $\{ \eps_{12}, \eps_{13}, \eps_{23} \}$
    \item basis for $T_3 = Q(-3)$: $\{ \eps_{123} \}$.
\end{enumerate}
 
The differentials for $\bfT$ are 
\[ 
\tau_1 =
\begin{blockarray}{*{4}{>{\scriptstyle}c}<{}}
1 & 2 & 3 \\ 
\begin{block}{[*{3}r]>{\scriptstyle}l<{}}
xy & xz & yz & \emptyset \\
\end{block}
\end{blockarray},
\: 
\tau_2 = 
\begin{blockarray}{*{4}{>{\scriptstyle}c}<{}}
12 & 13 & 23 \\ 
\begin{block}{[*{3}r]>{\scriptstyle}l<{}}
z & z & 0 & 1 \\ 
-y & 0 & y & 2 \\ 
0 & -x & -x & 3 \\
\end{block}
\end{blockarray},
\text{ and } 
\tau_3 = 
\begin{blockarray}{*{2}{>{\scriptstyle}c}<{}}
123 \\ 
\begin{block}{[r]>{\scriptstyle}l<{}}
1 & 12 \\
-1 & 13\\
1 & 23 \\
\end{block}
\end{blockarray}.
\] 

\subsection{A System of Homotopies}

Let $\fraka = \ideal{a_1, \ldots, a_c}$ be generated by a regular sequence $a_1, \ldots, a_c$ in $Q$ and $R = Q/\fraka$. The following definitions and constructions are due to Shamash \cite{sTPoinSerOALocalRing} in the $c=1$ case, and to Eisenbud \cite{eHomAlgOnACompleteInt, eEnrichedFreeResAChangeORings} when $c \geq 2$. We give a summary of the necessary background with some finer details omitted to suit our particular setting; a full treatment can be found in \cite[Chapter~3]{aInfFRes} \cite[Section~4.1]{epMinFResOCompleteInt}. 

\begin{definition}\label{def:sysofhomot}
Let $\bfG$ be a complex of free $Q$-modules. Let $\bfu = (u_1, \ldots, u_c)$, where each $u_i \geq 0$ is an integer; use $\bfet$ to denote the $t$th standard basis vector of length $c$ and $\bfzero$ to denote the zero vector. Set $|\bfu| = \sum_i u_i$. A \textit{system of higher homotopies $\sigma$} for $a_1, \ldots, a_c$ on $\bfG$ is a collection of maps
\[
\sigma_\bfu : \bfG \rightarrow \bfG[-2|\bfu|+1]
\]
on the modules in $\bfG$ where the three following conditions are satisfied: 
\begin{enumerate}[label=(\roman*)]
    \item \label{item:homot1} $\sigma_\bfzero$ is the differential on $\bfG$, 
    \item \label{item:homot2} for each $j \in [c]$, the map $\sigma_\bfzero \sigma_\bfej + \sigma_\bfej \sigma_\bfzero$ is multiplication by $a_j$ on $\bfG$, and 
    \item \label{item:homot3} if $|\bfu| \geq 2$, then 
    \[\sum\limits_{\bfb + \bfb'=\bfu} \sigma_\bfb \sigma_{\bfb'} = 0.\]
\end{enumerate}
\end{definition}

\begin{construction}
Consider the divided power algebra $Q\{y_1, \ldots, y_c\} \cong \oplus Qy_1^{(i_1)}y_2^{(i_2)}\cdots y_c^{(i_c)}$ on variables $y_1, \ldots, y_c$ of degree $2$, and the polynomial ring $Q[t_1, \ldots, t_c]$, where we treat the monomials in $Q\{y_1, \ldots, y_c\}$ and in $Q[t_1, \ldots, t_c]$ as dual to each other. In this way, we can prescribe an action of $Q[t_1, \ldots, t_c]$ on $Q\{y_1, \ldots, y_c\}$ by 
\[t_j y_j^{(i)} = y_j^{(i-1)}.\]
For a system of homotopies $\sigma$ on $\bfG$, the graded free $R$-module 
\[\Phi(\bfG) \coloneqq Q\{y_1, \ldots, y_c\} \otimes \bfG \otimes R \]
with differential 
\[ \varphi \coloneqq \sum t^\bfu \otimes \sigma_\bfu \otimes R\]
is called the \textit{Shamash construction} or \textit{Eisenbud--Shamash construction}.
\end{construction}

\begin{theorem} \cite{sTPoinSerOALocalRing, eHomAlgOnACompleteInt} If $\bfG$ is a $Q$-free resolution of a finitely generated module $N$, then $\Phi(\bfG)$ is an $R$-free resolution of $N$. 
\end{theorem}

More information on divided power algebras can be found in \cite[Appendix 2]{eCommAlgWAView}. For the reader who is satisfied to restrict to the characteristic zero case, the divided power algebra and the symmetric algebra behave roughly equivalently. 

\begin{example}
Let $Q = \bbk[x_1,x_2]$ and $R = Q/\ideal{x_1^5}$. We will use the Shamash construction to write a (minimal) resolution of $R/\ideal{x_1^2,x_2^2}$. The Taylor resolution 
\[\bfT: 0 \rightarrow Q(-4) \xrightarrow{\tau_2} Q(-2)^2 \xrightarrow{\tau_1} Q \rightarrow 0\]
resolves $Q/\ideal{x_1^2,x_2^2}$ minimally. The maps
\[\sigma_{1,0}: Q \xrightarrow{\barray{c}{x_1^3 \\ 0}} Q(-2)^2
\quad 
\text{ and }
\quad 
\sigma_{1,1}: Q(-2)^2 \xrightarrow{\barray{cc}{0 & -x_1^3}} Q(-4)\] 
form a system of homotopies on $\bfG$ for $x^5$, though from here on out, we will drop the first index from the subscript to avoid notational clutter. Identifying $Q(-d) \otimes R$ with $R(-d)$, we construct the Shamash resolution of $Q/\ideal{x_1^2,x_2^2}$ as an $R$-module: 
\[
\hspace{-.75in}
\cdots \rightarrow 
    \begin{array}{c}
        Qy^{(2)} \otimes R  \\
        \bigoplus \\
        Qy^{(1)} \otimes R(-4) 
    \end{array}
\xrightarrow{\barray{c|c}{\sigma_0 & \tau_2}}
Qy^{(1)} \otimes R(-2)^2
\xrightarrow{\barray{c}{\tau_1 \\ \hline \sigma_1}}
    \begin{array}{c}
        Qy^{(1)} \otimes R \\ 
        \bigoplus \\ 
        Q \otimes R(-4)
    \end{array}
\xrightarrow{\barray{c|c}{\sigma_0 & \tau_2}}
Q \otimes R(-2)^2
\xrightarrow{\barray{c}{\tau_1}}
Q \otimes R
\rightarrow 
0
\]
Note that $Qy^{(d)} \otimes R(-k)$ can be identified with $R(-5d-k)$, so we could in fact rewrite the constituent modules in the resolution above as $R(-5i+1) \oplus R(-5i)$ in homological degrees $2i$ and $R(-5i+3)^2$ in homological degrees $2i-1$, when $i>0$.
\end{example}

More generally, if $\deg(a_j) = d_j$, then we can identify $Qy_1^{(p_1)}y_2^{(p_2)}\cdots y_c^{(p_c)} \otimes Q(-k) \otimes R$ with $R(-k-\sum\limits_{i=1}^{c} p_id_i)$ in the Shamash construction to obtain the correct twist. 

\section{Taylor Resolutions Over Complete Intersections}

Let $I = \ideal{m_1, \ldots, m_r}$ be a monomial ideal in $Q$ and $\fraka = \ideal{a_1, \ldots, a_c}$ be generated by a regular sequence in $Q$ with $\fraka \subseteq I$. This means that each $a_i = f_{i,1}m_1 + \cdots + f_{i,r}m_r$ for some $f_1, \ldots, f_r \in Q$, though such an expression is not inherently unique without first fixing a monomial ordering on $Q$. Once these $\fij$ are fixed, the following construction will produce a system $\sigma$ of higher homotopies for $a_1, \ldots, a_c$ on $\bfT$, where $\bfT$ is the Taylor resolution of $Q/I$. 

\begin{definition}\label{def:taylorhomotopies}
Let $\bfu = (u_1, \ldots, u_c)$, where each $u_i \geq 0$ is an integer; use $\bfei$ to denote the $i$th standard vector of length $c$ and $\bfzero$ to denote the zero vector. Set $|\bfu| = \sum_i u_i$. Define a system of maps 
\[\sigma_u:\bfT \rightarrow \bfT[-2|\bfu|+1]\] 
in the following way: 
\begin{itemize}
    \item for $S = \{s_1, \ldots, s_k \} \subseteq [r]$,
    \[ \sigma_\bfzero(\eS) = \sum\limits_{s_i \in S} (-1)^{k-i} \frac{m_S}{m_{S-s_i}} \varepsilon_{S - s_i},\]
    that is, $\sigma_\bfzero = \tau$, 
    \item for $S = \{s_1, \ldots, s_k \} \subseteq [r]$ and $\barS = [r]-S$, \[\sigma_\bfei(\eS) = \sum\limits_{t \in \barS} (-1)^{k-p_{t,S}-1} \frac{f_{i,t} m_t m_S}{m_\St} \eSt \]
    where $p_{t,S}$ is the position of $t$ in $\St$, that is, $\St = \{s_1, \ldots, s_{p_{t,S}-1}, t, s_{p_{t,S}}, \ldots, s_k\}$, 
    \item and 
    \[ \sigma_{\mathbf{u}} = 0\] 
    for all other $\bfu \in \bbN^c$, i.e. if $|\bfu| \geq 2$. 
\end{itemize}
\end{definition}

\begin{theorem}\label{thm:systemofhomotopiesforT}
The maps $\sigma$ given in Definition~\ref{def:taylorhomotopies} form a system of higher homotopies for $a_1, \ldots, a_c$ on the Taylor complex $\bfT$. Therefore, when the Shamash construction is applied to the system of higher homotopies for $a_1, \ldots, a_c$ on $\bfT$, these maps furnish a resolution $\Phi$ of $R/I$. 
\end{theorem}

The existence of a system of higher homotopies is guaranteed by \cite{sTPoinSerOALocalRing, eHomAlgOnACompleteInt} Eisenbud and Shamash, so the main content of the theorem is, first, that no higher homotopies are necessary with the proposed $\sigma_\bfei$, and, second, that the terms in the differentials are a simple combination of the $f_{i,j}$ and, as in the Taylor resolution, a ratio of LCMs of the $m_i$. Therefore, once expressions for the $a_j$ in terms of the $m_i$ are chosen, the resulting resolution is extremely computable and completely prescribed -- no further choices are necessary. In the hypersurface case, this resolution automatically produces a matrix factorization that can be stated purely combinatorially, just in terms of subsets and least common multiples. Some overlap can be found in \cite{iFreeResAChangeORings} and \cite{abHomAlgModARegSeqWSpecialAttToCodimTwo}, but the result here uses different, less technical methods to provide more combinatorially-minded results.

\begin{proof}
The proof method is straightforward: the three criteria from Definition~\ref{def:sysofhomot} must be checked.

\ref{item:homot1} is satisfied by the definition of $\sigma_\bfzero$. 

To check \ref{item:homot3}, first note that $\sigma_\bfb  \sigma_{\bfb'} = 0$ if $\bfb + \bfb' > 2$, since, in this case, at least one of $|\bfb|, |\bfb'| > 2$. Therefore the only time that $\sum\limits_{\bfb+\bfb'=\bfu} \sigma_\bfb \sigma_{\bfb'}$ is not trivially zero is when $|\bfu|=2$. This occurs in two cases: when $\bfu = 2\bfei$ and when $\bfu = \bfei + \bfej$. 

In the case when $\bfu = 2 \bfei$, we have 
\[
\sum\limits_{\bfb + \bfb'= 2\bfei} \sigma_\bfb \sigma_{\bfb'} = \sigma_{2\bfei} \sigma_\bfzero + \sigma_\bfei \sigma_\bfei + \sigma_\bfzero \sigma_{2\bfei} = \sigma_\bfei \sigma_\bfei,
\]
so we must check that $\sigma_\bfei \sigma_\bfei = 0$. Indeed, we see that
\begin{align*}
\sigma_\bfei(\sigma_\bfei(\eS)) &= \sigma_\bfei \left( \sum\limits_{t \in \barS} (-1)^{k-\ptS-1} \frac{\fit m_t m_S}{m_\St} \eSt \right) \\ 
&= \sum\limits_{t \in \barS} (-1)^{k-\ptS-1} \frac{\fit m_t m_S}{m_\St} \sigma_\bfei(\eSt) \\ 
&= \sum\limits_{t \in \barS} (-1)^{k-\ptS-1} \frac{\fit m_t m_S}{m_\St} \left( \sum\limits_{q \in \barSt} (-1)^{k+1-p_{q,\St}-1} \frac{\fiq m_\St m_q}{m_\Stq} \eps_\Stq \right) \\ 
&= \sum\limits_{t \in \barS} \sum\limits_{q \in \barSt} (-1)^{-\ptS-p_{q,\St}-1} \frac{\fit \fiq m_t m_q m_S}{m_\Stq} \eps_\Stq \\ 
&= \sum\limits_{\substack{t<q \\ t,q \in \barS}} \left((-1)^{-\ptS - p_{q,\St}-1} + (-1)^{-\pqS - p_{t,\Sq}-1}\right)\frac{\fit \fiq m_t m_q m_S}{m_\Stq} \eps_\Stq.
\end{align*}

Because $t<q$, we have that $p_{t,\Sq} = \ptS$ but $p_{q,\St} = \pqS + 1$, so we can adjust the final sum to 
\[
\sum\limits_{\substack{t<q \\ t,q \in \barS}} \left((-1)^{-\ptS - (p_{q,S}+1)-1} + (-1)^{-\pqS - p_{t,S}-1}\right)\frac{\fit \fiq m_t m_q m_S}{m_\Stq} \eps_\Stq,
\]
which then reduces to $0$. 

In the case when $\bfu = \bfei + \bfej$, we have 
\[
\sum\limits_{\bfb + \bfb'= \bfei + \bfej} \sigma_\bfb \sigma_{\bfb'} = \sigma_{\bfei + \bfej} \sigma_\bfzero + \sigma_\bfei \sigma_\bfej + \sigma_\bfej \sigma_\bfei + \sigma_\bfzero \sigma_{\bfei + \bfej} = \sigma_\bfei \sigma_\bfej + \sigma_\bfej \sigma_\bfei,
\]
so we must check that $\sigma_\bfei \sigma_\bfej + \sigma_\bfej \sigma_\bfei =0$. First we see that 
\begin{align*}
\sigma_\bfei (\sigma_\bfej (\eS)) &= \sigma_\bfei \left( \sum\limits_{t \in \barS} (-1)^{k-\ptS-1} \frac{\fjt m_t m_S}{m_\St} \eSt \right) \\ 
&= \sum\limits_{t \in \barS} (-1)^{k-\ptS-1} \frac{\fjt m_t m_S}{m_\St} \sigma_\bfei(\eSt) \\ 
&= \sum\limits_{t \in \barS} (-1)^{k-\ptS-1} \frac{\fjt m_t m_S}{m_\St} \left(\sum\limits_{q \in \barSt} (-1)^{k+1-\pqSt -1} \frac{\fiq m_q m_\St}{m_\Stq} \eps_\Stq \right) \\
&= \sum\limits_{t \in \barS} \sum\limits_{q \in \barSt} (-1)^{-\ptS -\pqSt -1} \frac{\fjt \fiq m_t m_q m_S}{m_\Stq} \eStq \\
&= \sum\limits_{\substack{t<q \\ t,q \in \barS}} \left( (-1)^{-\ptS -\pqSt -1}\fjt \fiq + (-1)^{-\pqS -\ptSq-1} \fjq \fit \right) \frac{m_t m_q m_S}{m_\Stq} \eStq \\
&= \sum\limits_{\substack{t<q \\ t,q \in \barS}} \left( (-1)^{-\ptS -(\pqS+1) -1}\fjt \fiq + (-1)^{-\pqS -\ptS-1} \fjq \fit \right) \frac{m_t m_q m_S}{m_\Stq} \eStq \\ 
&= \sum\limits_{\substack{t<q \\ t,q \in \barS}} (-1)^{-\ptS-\pqS}\left( \fjt \fiq - \fjq \fit \right) \frac{m_t m_q m_S}{m_\Stq} \eStq.
\end{align*}

A similar computation shows that 
\begin{align*}
\sigma_\bfej(\sigma_\bfei(\eS)) 
&= \sum\limits_{\substack{t<q \\ t,q \in \barS}}(-1)^{-\ptS-\pqS} \left(\fit \fjq - \fjt \fiq \right) \frac{m_t m_q m_S}{m_\Stq} \eStq. 
\end{align*}
Combining the two computations, we see that $\sigma_\bfei \sigma_\bfej + \sigma_\bfej \sigma_\bfei = 0$. 

It remains to check \ref{item:homot2}.
We compute 
\begin{align*}
\sigma_\bfzero(\sigma_\bfej(\eS)) &= \sigma_\bfzero \left( \sum\limits_{t \in \barS} (-1)^{k-\ptS-1} \frac{\fjt m_t m_S}{m_\St} \eSt \right) \\ 
&= \sum\limits_{t \in \barS} (-1)^{k-\ptS-1}\frac{\fjt m_t m_S}{m_\St} \sigma_\bfzero\left(\eSt\right)\\ 
&= \sum\limits_{t \in \barS} (-1)^{k-\ptS-1} \frac{\fjt m_t m_S}{m_\St} \left( \sum\limits_{\substack{s_i < t \\ s_i \in \St}} (-1)^{k+1-i} \frac{m_\St}{m_{S \cup t - s_i}} \eps_{S \cup t - s_i} \right.\\
& \hspace{1.25in} \left. + (-1)^{k+1-\ptS} \frac{m_\St}{m_S} \eS + \sum\limits_{\substack{s_i > t \\ s_i \in \St}} (-1)^{k+1-(i+1)} \frac{m_\St}{m_{S \cup t - s_i}} \eps_{S \cup t - s_i}  \right) \\
&= \left( \sum\limits_{t \in \barS} \sum\limits_{\substack{s_i < t \\ s_i \in S}} (-1)^{-\ptS-i} \frac{\fjt m_t m_S}{m_{S \cup t - s_i}} \right) + \sum\limits_{t \in \barS} \fjt m_t \eS + \\ 
& \hspace{1.25in} + \left( \sum\limits_{t \in \barS} \sum\limits_{\substack{s_i > t \\ s_i \in S}} (-1)^{-\ptS-i-1} \frac{\fjt m_t m_S}{m_{S \cup t - s_i}} \eps_{S \cup t - s_i} \right) 
\end{align*}

Similarly we can see that
\begin{align*}
\sigma_\bfej(\sigma_\bfzero(\eS))
&= \left( \sum\limits_{s_i \in S} \sum\limits_{\substack{t<s_i \\ t \in \barS}} (-1)^{-i-\ptS} \frac{\fjt m_t m_S}{m_{S-s_i \cup t}} \eps_{S-s_i \cup t} \right) + \sum\limits_{s_i \in S} f_{j,s_i} m_{s_i} \eS \\ 
& \hspace{1.25in} + \left( \sum\limits_{s_i \in S} \sum\limits_{\substack{t > s_i \\ t \in \barS}} (-1)^{-i-\ptS-1} \frac{\fjt m_t m_S}{m_{S-s_i \cup t}} \eps_{S-s_i \cup t} \right).
\end{align*}

Adding the two expressions together, we see that the parenthesized double sums cancel, and we are left with 
\[
\sigma_\bfzero(\sigma_\bfej(\eS)) + \sigma_\bfej(\sigma_\bfzero(\eS)) = \sum\limits_{t \in \barS} \fjt m_t \eS + \sum\limits_{s_i \in S} f_{j,s_i} m_{s_i} \eS = \sum\limits_{i = 1}^r f_{j,i} m_i \eS = a_j \eS.
\]
\end{proof}

A few remarks should be made about the resolution $\bfF$. Due to its origin in the Taylor resolution, $\bfF$ is universal for all monomial ideals $I$ in $Q$ containing $\fraka$. The same uniformity that characterizes the Taylor resolution is also present in $\bfF$ -- these resolutions for two ideals generated by $r$ monomials will have the same ranks in all their free modules. However, this provides a quick bound to the Betti numbers of $I$. This bound could be calculated without explicit knowledge of the homotopies (see \cite[Proposition~3.3.5]{aInfFRes}) but, for the sake of completeness, we include the specific computation here. 

\begin{corollary}\label{cor:bettinumbers}
The total Betti numbers $\beta_i$ of $R/I$ for a monomial ideal $I$ in a complete intersection $R$ satisfy the inequality 
\[
\beta_{2m} \leq \sum\limits_{j=0}^m \binom{r}{2j} \binom{c+m-j-1}{c-1} 
\quad 
\text{ and }
\quad 
\beta_{2m+1} \leq \sum\limits_{j=0}^m \binom{r}{2j+1} \binom{c+m-j-1}{c-1}.
\]
\end{corollary}
\begin{proof}
This follows from a quick rank calculation of the modules in $F$ in the even and odd cases. To wit, 
\[
F_{2m} = \bigoplus\limits_{\substack{j=0 \\ |\bfu| = m-j}}^{m} Q\mathbf{y}^{(\mathbf{u})} \otimes T_{2j} \otimes R,
\]
where we use multinomial notation $\mathbf{y}^\bfu = y_1^{(u_1)} \cdots y_c^{(u_c)}$. Since $\rank T_{2j} = \binom{r}{2j}$ and there are $\binom{c+m-j-1}{c-1}$ monomials of degree $m-j$ in the $y$'s, we can see that 
\[
\rank F_{2m} = \sum\limits_{j=0}^m \binom{r}{2j} \binom{c+m-j-1}{c-1},
\]
which immediately gives the bound for $\beta_{2m}$. The bound for $\beta_{2m+1}$ can be found analogously.
\end{proof}

\begin{remark}\label{rmk:minimality}
Per \cite[Remark~(p.33)]{aInfFRes}, the minimality of $\bfF$ is closely linked to the minimality of $\bfT$, with an extra caveat: $\bfF$ is minimal if and only if $\bfT$ is minimal \textit{and} $\fraka \subseteq \frakm I$. In our particular maps, the necessity that $\fij \in \frakm$ comes from 
\[\sigma_\bfei(\eps_\emptyset) = \sum\limits_{t=1}^r (-1)^{0-1-1} \frac{\fit m_t m_\emptyset}{m_t} \eps_t = \sum\limits_{t=1}^r \fit \eps_t.\] 
Since the $\fit$ appear as entries in the maps of $\bfF$, they must be in $\frakm$ for $\bfF$ to be minimal. On the other hand, the coefficient of $\eSt$ in $\sigma_\bfei(\eS)$ is always a multiple of $\fit$, so $\fit \in \frakm$ is enough to ensure that $\bfF$ is minimal. 
\end{remark}

\begin{example}[Resolution of the Ground Field]
When $I=\ideal{x_1, \ldots, x_n}$, $\bfF$ recovers the Tate complex \cite{tHomONoethRingsALocalRings} resolving $\bbk$ over $R$, since the Taylor resolution coincides with the Koszul resolution of $\bbk$, and the Tate resolution can be obtained by applying the Shamash construction to the Koszul complex. By Remark~\ref{rmk:minimality}, the resolution is minimal, assuming $\deg a_i > 1$ for $a_1, \ldots, a_c$.
\end{example}

\begin{example}[Resolution of $\frakm^b$] 
If $\fraka$ is a regular sequence of degree $d$, then $\bfF$ gives a resolution of $\frakm^b$ for any power $1 \leq b \leq d$. This resolution is nonminimal when $b \geq 2$, since in this case the Taylor resolution does not minimally resolve $\frakm^b$ over $Q$.
\end{example}

The simplest case is when $\fraka$ is generated by a single monomial $a$. Here, at least one of the $m_i$ in $I$ must divide $a$, but it is more interesting to examine a case when more than one $m_i$ divides $a$ to see how the choice of expressing $a$ in terms of the $m_i$ affects the homotopy. 

\begin{example}[Codimension $1$ Monomial $\fraka$ Case] \label{ex:monhypersurface1}  Let $Q = \bbk[x,y,z]$, $a = xyz$, and $I = \langle xy, xz, yz \rangle$. The Taylor resolution of $I$ over $Q$ is given in Example~\ref{ex:taylor}.

Writing $xyz = z \cdot xy + 0 \cdot xz + 0 \cdot yz$, the homotopy $\sigma \coloneqq \sigma_{\mathbf 1}$ has the following behavior: 
\[
\sigma_0 = 
\begin{blockarray}{*{2}{>{\scriptstyle}c}<{}}
\emptyset \\
\begin{block}{[r]>{\scriptstyle}l<{}}
z & 1 \\
0 & 2 \\ 
0 & 3 \\
\end{block}
\end{blockarray}, 
\:
\sigma_1 = 
\begin{blockarray}{*{4}{>{\scriptstyle}c}<{}}
1 & 2 & 3 \\ 
\begin{block}{[*{3}r]>{\scriptstyle}l<{}}
0 & -xz & 0 & 12 \\
0 & 0 & -yz & 13 \\ 
0 & 0 & 0 & 23 \\
\end{block}
\end{blockarray},
\text{ and }
\sigma_2 = 
\begin{blockarray}{*{4}{>{\scriptstyle}c}<{}}
12 & 13 & 23 \\
\begin{block}{[*{3}r]>{\scriptstyle}l<{}}
0 & 0 & xyz & 123 \\
\end{block}
\end{blockarray}
\]

Therefore a resolution of $I$ over $R \coloneqq Q/\langle a \rangle$ is
\[
\bfF: \cdots \xrightarrow{\varphi_5} R(-6)^4 \xrightarrow{\varphi_4} \begin{array}{c}R(-5)^3 \\ \oplus \\ R(-3) \end{array} \xrightarrow{\varphi_3} R(-3)^4 \xrightarrow{\varphi_2} R(-2)^3 \xrightarrow{\varphi_1} R \rightarrow 0
\]
where 
$F_{2i} = R(-3i)^4$ and $F_{2i-1} = R(-3i+1)^3 \oplus R(-3(i-1))$ for $i \geq 2$, the initial maps are 
\[
\varphi_1 = 
    \begin{blockarray}{*{4}{>{\scriptstyle}c}<{}}
    1 & 2 & 3 \\ 
    \begin{block}{[*{3}r]>{\scriptstyle}c<{}}
    xy & xz & yz & \emptyset \\
    \end{block}
    \end{blockarray}
\text{ and }
\varphi_2 = 
    \begin{blockarray}{*{5}{>{\scriptstyle}c}<{}}
    \emptyset & 12 & 13 & 23 \\ 
    \begin{block}{[r|*{3}r]>{\scriptstyle}c<{}}
    z & z & z & 0 & 1 \\ 
    0 & -y & 0 & y & 2 \\ 
    0 & 0 & -x & -x & 3 \\
    \end{block}
    \end{blockarray},
\] 
and the (non-minimal) tail of the resolution is periodic with maps 
\[ 
\varphi_{2i-1} = 
    \begin{blockarray}{*{5}{>{\scriptstyle}c}<{}}
    1 & 2 & 3 & 123 \\ 
    \begin{block}{[*{3}r|r]>{\scriptstyle}c<{}}
    xy & xz & yz & 0 & \emptyset \\
    \BAhhline{----}
    0 & -xz & 0 & 1 & 12 \\
    0 & 0 & -yz & -1 & 13 \\ 
    0 & 0 & 0 & 1 & 23 \\
    \end{block}
    \end{blockarray}
\text{ and }
\varphi_{2i} = 
    \begin{blockarray}{*{5}{>{\scriptstyle}c}<{}}
    \emptyset & 12 & 13 & 23 \\ 
    \begin{block}{[r|*{3}r]>{\scriptstyle}c<{}}
    z & z & z & 0 & 1 \\ 
    0 & -y & 0 & y & 2 \\ 
    0 & 0 & -x & -x & 3 \\
    \BAhhline{----}
    0 & 0 & 0 & xyz & 123 \\
    \end{block}
    \end{blockarray}
\]
for $i \geq 2$. 

The vertical lines in the maps are visual dividers indicating the origin of the constituent blocks (whether they come from $\sigma_\bfzero$ or $\sigma_{\mathbf{1}}$). Note that the domains of $\varphi_{2i-1}$ and $\varphi_{2i}$ are indexed by the odd and even subsets of $\{1,2,3\}$ respectively. This division of domains into even and odd subsets will always appear in the tail of the resolution in the $c=1$ case.

Note that writing $xyz = 0 \cdot xy + y \cdot xz + 0 \cdot yz$, we get a different homotopy and therefore a resolution with the same modules as $\bfF$ but different differentials:
\[
\varphi'_2 = 
    \begin{blockarray}{*{5}{>{\scriptstyle}c}<{}}
    \emptyset & 12 & 13 & 23 \\ 
    \begin{block}{[r|*{3}r]>{\scriptstyle}c<{}}
    0 & z & z & 0 & 1 \\ 
    y & -y & 0 & y & 2 \\ 
    0 & 0 & -x & -x & 3 \\
    \end{block}
    \end{blockarray},
\varphi'_{2i-1} = 
    \begin{blockarray}{*{5}{>{\scriptstyle}c}<{}}
    1 & 2 & 3 & 123 \\ 
    \begin{block}{[*{3}r|r]>{\scriptstyle}c<{}}
    xy & xz & yz & 0 & \emptyset \\
    \BAhhline{----}
    xy & 0 & 0 & 1 & 12 \\
    0 & 0 & 0 & -1 & 13 \\ 
    0 & 0 & -yz & 1 & 23 \\
    \end{block}
    \end{blockarray}
\text{ and }
\varphi'_{2i} = 
    \begin{blockarray}{*{5}{>{\scriptstyle}c}<{}}
    \emptyset & 12 & 13 & 23 \\ 
    \begin{block}{[r|*{3}r]>{\scriptstyle}c<{}}
    0 & z & z & 0 & 1 \\ 
    y & -y & 0 & y & 2 \\ 
    0 & 0 & -x & -x & 3 \\
    \BAhhline{----}
    0 & 0 & -xyz & 0 & 123 \\
    \end{block}
    \end{blockarray}
\]
for $i \geq 2$. 

Finally, writing $xyz = 0 \cdot xy + 0 \cdot xz + x \cdot yz$ to yields a third 
resolution 
with maps
\[
\varphi''_2 = 
    \begin{blockarray}{*{5}{>{\scriptstyle}c}<{}}
    \emptyset & 12 & 13 & 23 \\ 
    \begin{block}{[r|*{3}r]>{\scriptstyle}c<{}}
    0 & z & z & 0 & 1 \\ 
    0 & -y & 0 & y & 2 \\ 
    x & 0 & -x & -x & 3 \\
    \end{block}
    \end{blockarray},
\varphi''_{2i-1} = 
    \begin{blockarray}{*{5}{>{\scriptstyle}c}<{}}
    1 & 2 & 3 & 123 \\ 
    \begin{block}{[*{3}r|r]>{\scriptstyle}c<{}}
    xy & xz & yz & 0 & \emptyset \\
    \BAhhline{----}
    0 & 0 & 0 & 1 & 12 \\
    xy & 0 & 0 & -1 & 13 \\ 
    0 & xz & 0 & 1 & 23 \\
    \end{block}
    \end{blockarray}
\text{ and }
\varphi''_{2i} = 
    \begin{blockarray}{*{5}{>{\scriptstyle}c}<{}}
    \emptyset & 12 & 13 & 23 \\ 
    \begin{block}{[r|*{3}r]>{\scriptstyle}c<{}}
    0 & z & z & 0 & 1 \\ 
    0 & -y & 0 & y & 2 \\ 
    x & 0 & -x & -x & 3 \\
    \BAhhline{----}
    0 & xyz & 0 & 0 & 123 \\
    \end{block}
    \end{blockarray}
\]
for $i \geq 2$. 

Examples~\ref{ex:monhypersurface1} shows that the universality of $\bfT$ does not quite carry over; there are some choices to be made that change the form of the resolution $\bfF$. However, characteristic of $\bbk$ permitting, we can create a sort of ``average homotopy'' $\frac{1}{3}(\sigma + \sigma' + \sigma'')$, which is \textit{also} a system of homotopies for $xyz$ on $\bfT$, and liberates us from making a choice of how to express $xyz$ in terms of the generators of $I$. It is straightforward to check that a convex combination of systems of higher homotopies will also create a system of higher homotopies, so this averaging scheme can be generalized. The idea of taking an average of all possibilities to create a choice-free resolution also appears in \cite{emoMinResOMonId}.
\end{example}

\begin{example}[Codimension $1$ Polynomial $\fraka$ Case] \label{ex:polyhypersurface}
Take $Q = \bbk[x,y]$, $I = \ideal{x^2, y^2}$, and $\fraka = \ideal{x^2y + xy^2}$, so that $f_{1,1}=y$ and $f_{1,2}=x$. Then the resolution of $I$ over $R = Q/\fraka$ is 
\[
\bfF: \cdots 
\xrightarrow{\varphi_5} 
R(-6) \oplus R(-7) \xrightarrow{\varphi_4}
R(-5)^2 \xrightarrow{\varphi_3} 
R(-3) \oplus R(-4) \xrightarrow{\varphi_2}
R(-2)^2 \xrightarrow{\varphi_1}
R \rightarrow 0 
\]
where 
\[
\varphi_1 =     \begin{blockarray}{*{3}{>{\scriptstyle}c}<{}}
    1 & 2 \\ 
    \begin{block}{[*{2}r]>{\scriptstyle}c<{}}
    x^2 & y^2 & \emptyset \\
    \end{block}
    \end{blockarray},
\quad 
\varphi_{2i} =     \begin{blockarray}{*{3}{>{\scriptstyle}c}<{}}
    \emptyset & 12 \\ 
    \begin{block}{[r|r]>{\scriptstyle}c<{}}
    y & y^2 & 1 \\
    x & -x^2 & 2 \\
    \end{block}
    \end{blockarray},
\text{ and } 
\varphi_{2i+1} = 
\begin{blockarray}{*{3}{>{\scriptstyle}c}<{}}
    1 & 2 \\ 
    \begin{block}{[rr]>{\scriptstyle}c<{}}
    x^2 & y^2 & \emptyset \\
    \BAhhline{--}
    x & -y & 12 \\
    \end{block}
    \end{blockarray}
\]
for $i \geq 1$.
\end{example}

\begin{example}[Codimension $2$ Polynomial $\fraka$ Case]\label{ex:polycodim2} 
Take $Q=\bbk[x,y,z,w]$, $I = \ideal{x^2, y^2, z^2, w^2}$ and $\fraka=\ideal{x^3+y^3, z^3+w^3}$. 
The homotopies for $x^3+y^3$ are 
\[
\hspace{-1.6cm}
\sigma_{\oz,0} = 
    \begin{blockarray}{*{2}{>{\scriptstyle}c}<{}}
    \emptyset \\ 
    \begin{block}{[r]>{\scriptstyle}l<{}}
    x & 1 \\ 
    y & 2 \\ 
    0 & 3 \\ 
    0 & 4 \\
    \end{block}
    \end{blockarray},
\sigma_{\oz,1} = 
    \begin{blockarray}{*{5}{>{\scriptstyle}c}<{}}
    1 & 2 & 3 & 4 \\ 
    \begin{block}{[*{4}r]>{\scriptstyle}l<{}}
    y & -x & 0 & 0 & 12 \\
    0 & 0 & -x & 0 & 13\\
    0 & 0 & 0 & -x & 14 \\
    0 & 0 & -y & 0 & 23 \\
    0 & 0 & 0 & -y & 24 \\ 
    0 & 0 & 0 & 0 & 34 \\
    \end{block}
    \end{blockarray},
\sigma_{\oz,2} = 
    \begin{blockarray}{*{7}{>{\scriptstyle}c}<{}}
    12 & 13 & 14 & 23 & 24 & 34\\ 
    \begin{block}{[*{6}r]>{\scriptstyle}l<{}}
    0 & -y & 0 & x & 0 & 0 & 123 \\ 
    0 & 0 & -y & 0 & x & 0 & 124 \\ 
    0 & 0 & 0 & 0 & 0 & x & 134 \\
    0 & 0 & 0 & 0 & 0 & y & 234 \\
    \end{block}
    \end{blockarray},
\sigma_{\oz,3} = 
    \begin{blockarray}{*{5}{>{\scriptstyle}c}<{}}
    123 & 124 & 134 & 234 \\ 
    \begin{block}{[*{4}r]>{\scriptstyle}l<{}}
    0 & 0 & y & -x & 1234 \\
    \end{block}
    \end{blockarray},
\]
while the homotopies for $z^3+w^3$ are 
\[
\hspace{-1.6cm}
\sigma_{\zo,0} = 
    \begin{blockarray}{*{2}{>{\scriptstyle}c}<{}}
    \emptyset \\ 
    \begin{block}{[r]>{\scriptstyle}l<{}}
    0 & 1 \\ 
    0 & 2 \\ 
    z & 3 \\ 
    w & 4 \\
    \end{block}
    \end{blockarray},
\sigma_{\zo,1} = 
    \begin{blockarray}{*{5}{>{\scriptstyle}c}<{}}
    1 & 2 & 3 & 4 \\ 
    \begin{block}{[*{4}r]>{\scriptstyle}l<{}}
    0 & 0 & 0 & 0 & 12 \\
    z & 0 & 0 & 0 & 13\\
    w & 0 & 0 & 0 & 14 \\
    0 & z & 0 & 0 & 23 \\
    0 & w & 0 & 0 & 24 \\ 
    0 & 0 & w & -z & 34 \\
    \end{block}
    \end{blockarray},
\sigma_{\zo,2} = 
    \begin{blockarray}{*{7}{>{\scriptstyle}c}<{}}
    12 & 13 & 14 & 23 & 24 & 34\\ 
    \begin{block}{[*{6}r]>{\scriptstyle}l<{}}
    z & 0 & 0 & 0 & 0 & 0 & 123 \\ 
    w & 0 & 0 & 0 & 0 & 0 & 124 \\ 
    0 & w & -z & 0 & 0 & 0 & 134 \\
    0 & 0 & 0 & w & -z & 0 & 234 \\
    \end{block}
    \end{blockarray},
\sigma_{\zo,3} = 
    \begin{blockarray}{*{5}{>{\scriptstyle}c}<{}}
    123 & 124 & 134 & 234 \\ 
    \begin{block}{[*{4}r]>{\scriptstyle}l<{}}
    w & -z & 0 & 0 & 1234 \\
    \end{block}
    \end{blockarray}.
\]

The modules in the resolution take the form 
\[
\bfF: \cdots 
\xrightarrow{\varphi_6} 
    \begin{array}{c}
    R(-8)^{12} \\ \bigoplus \\ R(-9)^8
    \end{array}
\xrightarrow{\varphi_5}
    \begin{array}{c}
    R(-6)^3 \\ \bigoplus \\ R(-7)^{12} \\ \bigoplus \\ R(-8) 
    \end{array}
\xrightarrow{\varphi_4}
    \begin{array}{c}
    R(-5)^8 \\ \bigoplus \\ R(-6)^4
    \end{array}
\xrightarrow{\varphi_3}
    \begin{array}{c}
    R(-3)^2 \\ \bigoplus \\ R(-4)^6
    \end{array} 
\xrightarrow{\varphi_2}
    R(-2)^4 
\xrightarrow{\varphi_1}
R \rightarrow 0
\]
and a schematic for the differentials is illustrated below. To save some notational hassle, we will shorten notation by using $\sigma$ and $\widetilde{\sigma}$ for $\sigma_\oz$ and $\sigma_\zo$, respectively. To visually distinguish the pieces, we color the blocks from $\tau$, $\sigma$, and $\tildesigma$ with blue, orange, and red, respectively. 

\begin{minipage}{0.4\textwidth}
    \begin{tikzpicture}[scale=.47]
    \draw[gray] (0,0) rectangle (4,1);
    
    \fill[blue, opacity=.3] (0,0) rectangle (4,1);
    
    \node (000100) at (2,0.5) {$\tau_1$};
    
    \node (phi1) at (-1,0.5) {$\varphi_1=$};
    
    \end{tikzpicture}
    
    \vspace{1.5cm}
    
	\begin{tikzpicture}[scale=.47] 
	%
	\draw[gray] (0,0) rectangle (8,4); 
	\draw[dashed] (1,0) -- (1,4)
	                (2,0) -- (2,4);
	\fill[orange, opacity=.5] (0,0) rectangle (1,4);
	\fill[red, opacity=.5] (1,0) rectangle (2,4);
	\fill[blue, opacity=.3] (2,0) rectangle (8,4);
	\node (100010) at (0.5,2) {$\sigma_0$};
	\node (100001) at (1.5,2) {$\tildesigma_0$};
	\node (100200) at (5,2) {$\tau_2$};
	\node (phi2) at (-1,2) {$\varphi_2=$}; 
	\end{tikzpicture}
\end{minipage}
\begin{minipage}{0.6\textwidth} 
	\begin{tikzpicture}[scale=.47] 
	
	\draw[gray] (0,0) rectangle (16,12); 
	
	\draw[dashed] (1,4) -- (1,12)
	                (2,4) -- (2,12)
	                (3,0) -- (3,12)
	                (9,0) -- (9,12)
	                (15,0) -- (15,8)
	                (1,4) -- (16,4)
	                (0,8) -- (15,8);
	
	\fill[opacity=.5, orange] (0,8) rectangle (1,12); 
	\fill[opacity=.5, red] (1,8) rectangle (2,12);
	\fill[opacity=.3, blue] (3,8) rectangle (9,12);
	
	\fill[opacity=.5, orange] (1,4) rectangle (2,8); 
	\fill[opacity=.5, red] (2,4) rectangle (3,8);
	\fill[opacity=.3, blue] (9,4) rectangle (15,8);
	
	\fill[opacity=.5, orange] (3,0) rectangle (9,4);
	\fill[opacity=.5, red] (9,0) rectangle (15,4);
	\fill[opacity=.3, blue] (15,0) rectangle (16,4);
	
	\node (110020) at (0.5,10) {$\sigma_0$};
	\node (110011) at (1.5,10) {$\tildesigma_0$};
	\node (110210) at (6,10) {$\tau_2$}; 
	\node (101011) at (1.5,6) {$\sigma_0$};
	\node (101002) at (2.5,6) {$\tildesigma_0$};
	\node (101201) at (12,6) {$\tau_2$};
	\node (300210) at (6,2) {$\sigma_2$};
	\node (300201) at (12,2) {$\tildesigma_2$};
	\node (300400) at (15.5,2) {$\tau_4$};
	
	\node (phi4) at (-1,6) {$\varphi_4=$};
	
	\end{tikzpicture}
\end{minipage}

\begin{minipage}{0.4\textwidth}
	\begin{tikzpicture}[scale=.47] 
	
	\draw[gray] (0,0) rectangle (12,8); 
	
	\draw[dashed] (4,0) -- (4,8)
	                (8,0) -- (8,7)
	                (0,6) -- (12,6)
	                (0,7) -- (8,7);
	
	\fill[blue, opacity=.3] (0,7) rectangle (4,8);
	\fill[blue, opacity=.3] (4,6) rectangle (8,7);
	
	\fill[orange, opacity=.5] (0,0) rectangle (4,6);
	\fill[red, opacity=.5] (4,0) rectangle (8,6);
	\fill[blue, opacity=.3] (8,0) rectangle (12,6);
	
	\node (010110) at (2,7.5) {$\tau_1$};
	\node (200110) at (2,3.5) {$\sigma_1$};
	\node (001101) at (6,6.5) {$\tau_1$};
	\node (200101) at (6,3.5) {$\tildesigma_1$};
	\node (200300) at (10,3.5) {$\tau_3$};
	
	\node (phi3) at (-1,4) {$\varphi_3=$};

    \end{tikzpicture}
\end{minipage}
\begin{minipage}{0.6\textwidth}
	\centering
	\begin{tikzpicture}[scale=.47]
	
	\draw[gray] (0,0) rectangle (20,16);
	
	\draw[dashed] (0,15) -- (8,15)
	                (4,14) -- (12,14)
	                (0,13) -- (16,13)
	                (0,7) -- (20,7)
	                (4,1) -- (20,1)
	                (4,14) -- (4,16)
	                (4,1) -- (4,13)
	                (8,1) -- (8,15)
	                (12,0) -- (12,14)
	                (16,0) -- (16,13);
	
	\fill[opacity=.3,blue] (0,15) rectangle (4,16);
	\fill[blue, opacity=.3] (4,14) rectangle (8,15);
	\fill[blue, opacity=.3] (8,13) rectangle (12,14); 
	
	\fill[orange, opacity=.5] (0,7) rectangle (4,13);
	\fill[red, opacity=.5] (4,7) rectangle (8,13);
	\fill[blue, opacity=.3] (12,7) rectangle (16,13);
	
	\fill[orange, opacity=.5] (4,1) rectangle (8,7);
	\fill[red, opacity=.5] (8,1) rectangle (12,7); 
	\fill[blue, opacity=.3] (16,1) rectangle (20,7);
	
	\fill[orange, opacity=.5] (12,0) rectangle (16,1);
	\fill[red, opacity=.5] (16,0) rectangle (20,1);
	
	\node (020120) at (2,15.5) {$\tau_1$};
	\node (011111) at (6,14.5) {$\tau_1$};
	\node (002102) at (10,13.5) {$\tau_1$};
	
	\node (210120) at (2,10) {$\sigma_1$};
	\node (210111) at (6,10) {$\tildesigma_1$};
	\node (210310) at (14,10) {$\tau_3$};
	
	\node (201111) at (6,4) {$\sigma_1$};
	\node (201102) at (10,4) {$\tildesigma_1$};
	\node (201301) at (18,4) {$\tau_3$};
	
	\node (400310) at (14,0.5) {$\sigma_3$};
	\node (400301) at (18,0.5) {$\tildesigma_3$};
	
	\node (phi5) at (-1,8) {$\varphi_5=$}; 
	
	\end{tikzpicture}
\end{minipage}

\end{example}

\bibliographystyle{alpha}
\bibliography{biblio}

\end{document}